\documentclass{amsart} 
\usepackage{amssymb,times, amscd,amsmath,amsthm, xypic}
\usepackage{geometry}
\author{Vittoria Bussi$^{(*)}$ and Shoji Yokura$^{(**)}$}

\address
{Vittoria Bussi: The Mathematical Institute, 24-29 St. Giles, Oxford, OX1 3LB, U.K.}
\email{bussi@maths.ox.ac.uk}

\address
{Shoji Yokura: Department of Mathematics and Computer Science, 
Faculty of Science, 
Kagoshima University, 21-35 Korimoto 1-chome, Kagoshima 890-0065, Japan}
\email {yokura@sci.kagoshima-u.ac.jp}

\date{}

\thanks{(*) Funded by EPSRC\\
\quad (**) Partially supported by JSPS KAKENHI Grant Number 24540085}

\title[Naive motivic Donaldson--Thomas type Hirzebruch classes]
{Naive motivic Donaldson--Thomas type Hirzebruch classes \\ 
and some problems }

\begin{document} 
\numberwithin{equation}{section}
\newtheorem{thm}[equation]{Theorem}
\newtheorem{pro}[equation]{Proposition}
\newtheorem{prob}[equation]{Problem}
\newtheorem{qu}[equation]{Question}
\newtheorem{cor}[equation]{Corollary}
\newtheorem{con}[equation]{Conjecture}
\newtheorem{lem}[equation]{Lemma}
\theoremstyle{definition}
\newtheorem{ex}[equation]{Example}
\newtheorem{defn}[equation]{Definition}
\newtheorem{ob}[equation]{Observation}
\newtheorem{rem}[equation]{Remark}
\renewcommand{\rmdefault}{ptm}
\def\alp{\alpha}
\def\be{\beta}
\def\jeden{1\hskip-3.5pt1}
\def\om{\omega}
\def\bigstar{\mathbf{\star}}
\def\ep{\epsilon}
\def\vep{\varepsilon}
\def\Om{\Omega}
\def\la{\lambda}
\def\La{\Lambda}
\def\si{\sigma}
\def\Si{\Sigma}
\def\Cal{\mathcal}
\def\m {\mathcal}
\def\ga{\gamma}
\def\Ga{\Gamma}
\def\de{\delta}
\def\De{\Delta}
\def\bF{\mathbb{F}}
\def\bH{\mathbb H}
\def\bPH{\mathbb {PH}}
\def \bB{\mathbb B}
\def \bA{\mathbb A}
\def \bC{\mathbb C}
\def \bOB{\mathbb {OB}}
\def \bM{\mathbb M}
\def \bOM{\mathbb {OM}}
\def \mA{\mathcal A}
\def \mB{\mathcal B}
\def \mC{\mathcal C}
\def \mR{\mathcal R}
\def \mH{\mathcal H}
\def \mM{\mathcal M}
\def \mM{\mathcal M}
\def \mTOP{\mathcal {TOP}}
\def \mAB{\mathcal {AB}}
\def \bK{\mathbb K}
\def \bG{\mathbf G}
\def \bL{\mathbb L}
\def\bN{\mathbb N}
\def\bR{\mathbb R}
\def\bP{\mathbb P}
\def\bZ{\mathbb Z}
\def\bC{\mathbb  C}
\def \bQ{\mathbb Q}
\def\op{\operatorname}

\maketitle

\begin{abstract}
Donaldson-Thomas invariant is expressed as the weighted Euler
 characteristic of the so-called Behrend (constructible) function. In \cite{Behrend} 
Behrend introduced a DT-type invariant for a morphism. Motivated by this invariant, we extend the motivic Hirzebruch class to naive Donaldson-Thomas type analogues. We also discuss a categorification of the DT-type invariant for
a morphism from a bivariant-theoretic viewpoint, and we finally pose some related questions for further investigations.
\end{abstract}


\section{Introduction}

The Donaldson--Thomas invariant $\chi^{DT}(\m M)$ (abbr. DT invariant) is the virtual count of the moduli space $\m M$ of stable coherent sheaves on a Calabi--Yau threefold over $k$. Here $k$ is an algebraically closed field of characteristic zero. Foundational materials for DT invariants can be found in \cite{Thomas}, \cite{Behrend}, \cite{JS}, \cite{KS1}. In \cite{Behrend} Behrend made an important observation that the Donaldson--Thomas invariant $\chi^{DT}(\m M)$ is described as the weighted Euler characteristic $\chi(\m M, \nu_{\m M})$ of the so-called Behrend (constructible) function $\nu_{\m M}$.
For a scheme $X$ of finite type, the Donaldson--Thomas type invariant $\chi^{DT}(X)$ is defined as $\chi(X, \nu_X)$.
The topological Euler characteristic (more precisely, the topological Euler characteristic with compact support) $\chi$ satisfies the scissor formula
$\chi(X) = \chi(Z) + \chi(X \setminus Z)$
for a closed subvariety $Z \subset X$. This scissor formula implies that $\chi$ can be considered as the homomorphism from the Grothendieck group of varieties 
$\chi:K_0(\m V) \to \bZ$, and furthermore it can be extended to the relative Grothendieck group, $\chi:K_0(\m V/X) \to \bZ$ for each scheme $X$. The Grothendieck--Riemann--Roch version of the homomorphism $\chi:K_0(\m V/X) \to \bZ$ is the motivic Chern class transformation
${T_{-1}}_*:K_0(\m V/X) \to H_*^{BM}(X)\otimes \bQ$. Namely we have that
\begin{itemize}
\item When $X$ is a point, ${T_{-1}}_*:K_0(\m V/X) \to H_*^{BM}(X)\otimes \bQ$ equals the homomorphism $\chi:K_0(\m V) \to \bZ \hookrightarrow \bQ$.
\item The composite $\int_X \circ \,  {T_{-1}}_* = \chi: K_0(\m V/X) \to \bZ \hookrightarrow \bQ$. 
\end{itemize}
Here ${T_{-1}}_*:K_0(\m V/X) \to H_*^{BM}(X)\otimes \bQ$ is the specialization to $y = -1$ of the motivic Hirzebruch class transformation ${T_y}_*:K_0(\m V/X) \to H_*^{BM}(X)\otimes \bQ[y]$ (see \cite{BSY}).\\

On the other hand the Donaldson--Thomas type invariant $\chi^{DT}(X)$ does not in general satisfy the scissor formula
$\chi^{DT}(X) \not = \chi^{DT}(Z) + \chi^{DT}(X \setminus Z).$ Namely, $\chi^{DT}(-)$ cannot be captured as a homomorphism $\chi^{DT}:K_0(\m V) \to \bZ.$
Instead the following scissor formula holds:
\begin{equation}\label{scissor}
\chi^{DT}(X \xrightarrow {\op{id}_X} X)  = \chi^{DT}(Z \xrightarrow {i_{Z,X}} X) + \chi^{DT}(X \setminus Z \xrightarrow {i_{X \setminus Z,X}} X).
\end{equation}
Here $i_{Z,X}$ and  $i_{X \setminus Z,X}$ are the inclusions. For this formula to make sense, we need 
the Donaldson--Thomas type invariant $\chi^{DT}(X \xrightarrow f Y)$ for a morphism $f:X \to Y$, which is also introduced in \cite{Behrend} and simply defined as $\chi(X, f^*\nu_Y)$. Then, $\chi^{DT}$ can be considered as the homomorphism
$\chi^{DT}:K_0(\m V/X) \to \bZ.$ Note that in the case when $X$ is a point, $\chi^{DT}:K_0(\m V/pt) = K_0(\m V) \to \bZ$ is the usual Euler characteristic homomorphism $\chi:K_0(\m V) \to \bZ.$ 

In this paper we consider Grothendieck--Riemann--Roch type formulas for $\chi^{DT}$, using the motivic Hirzebruch class transformation ${T_y}_*$ (\cite{BSY}), and we also propose to consider a bivariant-theoretic aspect for the ``categorification" of the DT invariant. By that we mean a graded vector space encoding an appropriate cohomology theory whose Euler characteristic is equal to DT invariant. Naive reasons for the latter one are the following. The categorification of the Euler characteristic is nothing but 
$$\chi(X) := \sum_i (-1)^i \op{dim}_{\bR} H^i_c(X; \bR).$$ 
Note that the compact-support-cohomology $H^i_c(X; \bR)$ is isomorphic to the Borel--Moore homology  $H^{BM}_i(X; \bR)$.  The categorification of the Hirzebruch $\chi_y$-genus is 
$$\chi_y(X) = \sum (-1)^i \op{dim}_{\bC} Gr^p_F (H^i_c(X;\bC))(-y)^p$$
 with $F$ being the Hodge filtration of the mixed Hodge structure of $H^i_c(X;\bC)$.
 Since the DT type invariant of a morphism satisfies the scissor formula (\ref{scissor}), we propose to introduce some bivariant-theoretic homology theory $\Theta^*(X \xrightarrow f Y)$ ``categorifying" $\chi^{DT}(X \xrightarrow f Y)$, that is $\chi^{DT}(X \xrightarrow f Y)=\sum_i (-1)^i \op {dim} \Theta^i(X \xrightarrow f Y)$. (Here we denote it ``symbolically"; as described in the case of $\chi_y$-genus, the above alternating sum of the dimensions might be complicated involving some other ingredients such as mixed Hodge structures.)
\section{Donaldson--Thomas type invariants of morphisms}


Let $\frak K$ be an algebraically closed field of characteristic $p$, which is not necessarily zero. 
Let $X$ be a $\frak K$-scheme of finite type. For a prime number $\ell$ such that $\ell \not = p$ and the field $\bQ_{\ell}$ of $\ell$-adic numbers, the following Euler characteristic
$$\chi(X):= \sum_i (-1)^i \op{dim}_{\bQ_{\ell}}H_c^i(X, \bQ_{\ell})$$
is independent on the choice of the prime number $\ell$. In fact the following properties hold (e.g., see \cite[Theorem 3.10]{Joyce}):

\begin{thm}\label{basic} Let $\frak K$ be an algebraically closed field and $X, Y$ be separated
$\frak K$ -schemes of finite type. Then
\begin{enumerate}
\item  If $Z$ is a closed subscheme of $X$, then $\chi(X) = \chi(Z) + \chi(X \setminus Z)$.
\item  $\chi(X \times Y ) = \chi(X) \chi(Y)$.
\item $\chi(X)$ is independent of the choice of $\ell$ in the above definition
\item If $\frak K = \bC$, $\chi(X)$ is the usual Euler characteristic with the analytic topology.
\item  $\chi(\frak K^m) = 1$ and $\chi(\frak K\bP ^m) = m+ 1$ for $\forall m>0$
\end{enumerate}
\end{thm}

For a constructible function $\alp: X \to \bZ$ on $X$ the weighted Euler characteristic $\chi (X, \alp)$ is defined by
$$\chi(X, \alp):= \sum_m m \chi(\alp^{-1}(m)).$$


Let $X$ be embeddable in a smooth scheme $M$ and let $C_MX$ be the normal cone of $X$ in $M$ and let $\pi: C_MX \to X$ be the projection and $C_MX = \sum m_iC_i$, where $m_i \in \mathbb Z$ are multiplicities and $C_i$'s are irreducible components of the cycle. Then the following cycle
$$\frak C_{X/M} := \sum (-1)^{dim(\pi(C_i))} m_i \pi(C_i) \in \mathcal Z(X)$$
is in fact independent of the choice of the embedding of $X$ into a smooth $M$, thus simply denoted by $\frak C_X$ without referring to the ambient smooth $M$ and is called the distinguished cycle of the scheme. Then consider the isomorphism from the abelian groups $\m Z(X)$ of cycles to the abelian group $\m F(X)$ of constructible functions
$$\op {Eu} : \m Z(X) \xrightarrow {\cong} \m F(X)$$
which is defined by $\op {Eu}(\sum_i m_i [Z_i]) := \sum_i m_i \op {Eu}_{Z_i}$, where $\op {Eu}_Z$ denotes the local Euler obstruction supported on the subscheme $Z_i$. Then the image of the distinguished cycle $\frak C_X$ under the above isomorphism $\op {Eu}$ defines a canonical integer valued constructible function
$$\nu_X := \op {Eu}(\frak C_X),$$
which is called the \emph{Behrend} function. The fundamental properties of the Behrend function are the following.
\begin{thm}
\begin{enumerate}

\item For a smooth point $x$ of a scheme $X$ of dimension $n$, $\nu_X(x) =(-1)^n$. In particular, if $X$ is smooth of dimension $n$, then $\nu_X= (-1)^n\jeden_X$.
\item $\nu_{X \times Y}  = \nu_X \nu_Y$. 
\item If $f: X \to Y$  is smooth of relative dimension $n$, then $\nu_X = (-1)^n f^* \nu_Y$ . 
\item In particular, if $f: X \to Y$ is \'etale, then  $\nu_X = f^* \nu_Y$. \\
\end{enumerate}
\end{thm}

The weighted Euler characteristic of the above Behrend function is called the \emph{Donaldson--Thomas type invariant} and denoted by $\chi^{DT}(X)$:
$$\chi^{DT}(X):= \chi(X, \nu_X).$$

In \cite[Definition 1.7]{Behrend} Kai Behrend defined the following.

\begin{defn}
The \emph{DT-invariant} or \emph{virtual count} of a morphism $f:X \to Y$ is defined by
$$\chi^{DT}(X \xrightarrow {f} Y) := \chi(X, f^*\nu_Y),$$\
where $\nu_Y$ is the Behrend function of the target scheme $Y$. 
\end{defn}

\begin{rem} Here we emphasize that $\chi^{DT}(X \xrightarrow {f} Y)$ is defined by the constructible function $f^*\nu_Y$ \emph {on the source scheme $X$}. From the definition we can observe the following:
\begin{enumerate}
\item $\chi^{DT}(X \xrightarrow {\op{id}_X} X) = \chi(X,\nu_X) = \chi^{DT}(X)$ is the DT-invariant of $X$.
\item $\chi^{DT}(X \xrightarrow {\pi_X} pt) = \chi(X,f^*\nu_{pt}) = \chi(X, \jeden_X) = \chi(X)$ is the topological Euler-Poincar\'e characteristic of $X$.
\item If $Y$ is \emph{smooth}, whatever the morphism $f:X \to Y$ is, we have
$$\chi^{DT}(X \xrightarrow {f} Y) = (-1)^{\op{dim}Y} \chi(X).$$
The very special case is that $Y = pt$, which is the above (2). \\
\end{enumerate}
\end{rem}

The Euler characteristic $\chi(-)$ satisfies the additivity $\chi(X) = \chi(Z) + \chi(X \setminus Z)$ for a closed subscheme $Z \subset X$. Hence, $\chi$ is considered as a homomorphism from the Grothendieck group of varieties $\chi:K_0(\m V) \to \bZ$ and furthermore as a homomorphism from the relative Grothendieck group of varieties over a fixed variety $X$ (\cite{Looijenga}) 
$$\chi: K_0(\m V/X) \to \bZ,$$
which is defined by $\chi([V \xrightarrow h X]) = \chi(V) = \chi(V, \jeden_V) = \chi(V,h^*\jeden_X) = \chi(X,h_*\jeden_V)$. Moreover, the following diagram commutes:
\begin{equation}\label{diagram}
\xymatrix{
K_0(\m V/X)\ar[dr]_ {\chi}\ar[rr]^ {f_*} && K_0(\m V/Y)\ \ar[dl]^{\chi}\\
& \bZ.}
\end{equation}
On the other hand we have  that 
$\chi^{DT}(X) \not = \chi^{DT}(Z) + \chi^{DT}(X \setminus Z).$
Thus $\chi^{DT}(-)$ cannot be captured as a homomorphism $\chi^{DT}:K_0(\m V) \to \bZ.$
However, we have that
$$\chi^{DT}(X \xrightarrow {\op{id}_X} X)  = \chi^{DT}(Z \xrightarrow {i_{Z,X}} X) + \chi^{DT}(X \setminus Z \xrightarrow {i_{X \setminus Z,X}} X).$$ 

\begin{lem} If we define $\chi^{DT}([V \xrightarrow h X]):= \chi(V, h^*\nu_X)$, then we get the homomorphism $\chi^{DT}:K_0(\m V/X) \to \bZ$.
\end{lem}
\begin{proof} Clearly the definition $\chi^{DT}([V \xrightarrow h X]):= \chi(V, h^*\nu_X)$ is 
independent of the choice of the representative of the isomorphism class $[V \xrightarrow h X]$. For a closed subvariety $W \subset V$, we have
\begin{align*}
\chi^{DT}([V \xrightarrow h X] &= \chi(V, h^*\nu_X)\\
& = \chi(W, h^*\nu_X) + \chi(V \setminus W, h^*\nu_X)\\
& = \chi(W, h_{|W}^*\nu_X) + \chi(V \setminus W, h_{|V \setminus W}^*\nu_X)\\
& = \chi^{DT}([W \xrightarrow {h_{|W}} X]) + \chi^{DT}([W \xrightarrow {h_{|V \setminus W}} X]).
\end{align*}
Thus we get the homomorphism $\chi^{DT}:K_0(\m V/X) \to \bZ$.
\end{proof}

\begin{lem} If $f:X \to Y$ satisfies the condition that $\nu_X = f^*\nu_Y$ (such a morphism shall be called a ``Behrend morphism") , then the following diagram commutes:
$$\xymatrix{
K_0(\m V/X)\ar[dr]_ {\chi^{DT}}\ar[rr]^ {f_*} && K_0(\m V/Y)\ \ar[dl]^{\chi^{DT}}\\
& \bZ.}$$
\end{lem}
\begin{proof} It is straightforward:
\begin{align*}
\chi^{DT} \circ f_* ([V \xrightarrow h X]) & = \chi^{DT}([V \xrightarrow {f \circ h} X]) \\
& = \chi(V, (f \circ h)^*\nu_Y) \\
& = \chi(V, h^*f^*\nu_Y) \\
& = \chi(V, h^*\nu_X) \quad \text {(since $\nu_X = f^*\nu_Y$)}\\
& = \chi^{DT}([V \xrightarrow h X]).
\end{align*}.
\end{proof}
\begin{rem} 
An \'etale map is a typical example of a Behrend morphism. 
\end{rem}

\begin{rem}For a general morphism $f: X \to Y$, we have that $f^*\nu_Y = (-1)^{\op {reldim} f} \nu_X + \Theta(X_{sing} \cup f^{-1}(Y_{sing}))$, where $\op {reldim} f := \op {dim} X - \op {dim}Y$ is the relative dimension of $f$ and $\Theta(X_{sing} \cup f^{-1}(Y_{sing}))$ is some constructible functions supported on the singular locus $X_{sing}$ of $X$ and the inverse image of the singular locus $Y_{sing}$ of $Y$. 
As 
$$\nu_X = (-1)^{\op {dim} X} \jeden_X + \text {some constructible function supported on $X_{sing}$},$$
then
$$f^*\nu_Y = (-1)^{\op {dim}X} f^*\jeden_Y + f^*(\text {some constructible function supported on $Y_{sing}$}).$$ Hence in general we have
$$(\chi^{DT}\circ f_*)([V \xrightarrow h X]) = (-1)^{\op {reldim} f}\chi^{DT}([V \xrightarrow h X]) + \text{extra terms}.$$
\end{rem}

To avoid taking care of the sign, let us introduce the twisted Behrend function 
$$\widetilde \nu_X := (-1)^{\op {dim} X} \nu_X,$$
which will be used later again. Note that if $X$ is smooth, $\widetilde \nu_X = \jeden_X$. Then we define the twisted Donaldson--Thomas type invariant $\widetilde {\chi}^{DT}(X)$ by
$\widetilde {\chi}^{DT}(X \xrightarrow f Y) := \chi(X, f^*\widetilde \nu_Y).$
Then for a morphism $f:X \to Y$ we have
$f^*\widetilde \nu_Y = \widetilde \nu_X + \widetilde \Theta(X_{sing} \cup f^{-1}(Y_{sing})).$
In particular the above lemma is modified as follows:
\begin{lem}
 If $f:X \to Y$ satisfies the condition that $\widetilde\nu_X = f^*\widetilde\nu_Y$ (such a morphism shall be called a ``twisted Behrend morphism"; a smooth morphism is a typical example for $\widetilde\nu_X = f^*\widetilde\nu_Y$) , then the following diagram commutes:
$$\xymatrix{
K_0(\m V/X)\ar[dr]_ {\widetilde \chi^{DT}}\ar[rr]^ {f_*} && K_0(\m V/Y)\ \ar[dl]^{\widetilde \chi^{DT}}\\
& \bZ.}$$
\end{lem}


\section{Generalized Donaldson--Thomas type invariants of morphisms}
Mimicking the above definition of $\chi^{DT}(X \xrightarrow {f} Y)$ and ignoring the geometric or topological interpretation, we define the following.

\begin{defn} For a morphism $f:X \to Y$ and a constructible function $\delta_Y \in \m F(Y)$ we define
$$\chi^{\delta_Y}(X \xrightarrow {f} Y):= \chi(X, f^*\delta_Y).$$
\end{defn}

\begin{lem} For a morphism $f:X \to Y$ and a constructible function $\alp \in \m F(X)$ we have
$$\chi(X, \alp) = \chi(Y, f_*\alp).$$
\end{lem}

\begin{cor} For a morphism $f:X \to Y$ and a constructible function $\delta_Y \in \m F(Y)$ we have
$$\chi(X, f^*\delta_Y) = \chi(Y, f_*f^*\delta_Y).$$
\end{cor} 

\begin{rem} For the constant map $\pi_X : X \to pt$, the pushforward homomorphism
$${\pi_X}_*:\m F(X) \to \m F(pt) = \bZ$$
is nothing but the fact that ${\pi_X}_*(\alp) = \chi(X, \alp)$ (by the definition of the pushforward).
Hence, the above equality $\chi(X, \alp) = \chi(Y, f_*\alp)$ is paraphrased as the commutativity of the following diagram:

$$\xymatrix{
\m F(X) \ar[dr]_ {{\pi_X}_*}\ar[rr]^ {f_*} && \m F(Y) \ar[dl]^{{\pi_Y}_*}\\
& \m F(pt)= \bZ.}$$
Namely, ${\pi_X}_* = (\pi_Y \circ f) _* = {\pi_Y} \circ f_*.$ This might suggest that $\m F(-)$ is a covariant functor, but we need to be a bit careful. In fact, $\m F(-)$ is certainly a covariant functor \emph{provided that the ground field $\frak K$ is of characteristic zero}. However, if it is not of characteristic zero, then it may happen that $(g \circ f)_* \not = g_* \circ f_*$, for which see Sch\"urmann's example in \cite{Joyce}.\\
\end{rem}

\begin{rem} If we define $\jeden_*:K_0(\m V/X) \to \m F(X)$ by $\jeden_*([V \xrightarrow h X]) := h_* \jeden_V$, then for a morphism $f:X \to Y$ we have the following commutative diagrams:
$$\xymatrix{
 K_0(\m V/X) \ar[d]_{\jeden_*} \ar[rr]^ {f_*} && K_0(\m V/Y) \ar[d]^{\jeden_*}\\
 \m F(X) \ar[dr]_ {{\pi_X}_*}\ar[rr]_ {f_*} && \m F(Y) \ar[dl]^{{\pi_Y}_*}\\
& \m F(pt)= \bZ.}$$
$({\pi_X}_* \circ \jeden_*)([V \xrightarrow h X]) = \chi([V \xrightarrow h X])$ and the outer triangle is nothing but the commutative diagram (\ref{diagram}) mentioned before.
\end{rem}

Here we 
emphasize that the above equality $\chi(X, f^*\delta_Y) = \chi(Y, f_*f^*\delta_Y)$ have the following two aspects:

\begin{itemize}
\item The invariant on LHS for a morphism $f:X \to Y$ is defined on the source space $X$.
\item The invariant on RHS for a morphism $f:X \to Y$ is defined on the target space $Y$.
\end{itemize}

So, in order to emphasize the difference, we introduce the following notation:
$$\chi^{\delta_Y}(X \xrightarrow f Y)^X:= \chi(X, f^*\delta_Y) = \chi(Y, f_*f^*\delta_Y) =:\chi^{\delta_Y}(X \xrightarrow f Y)_Y.$$

Since we want to deal with higher class versions of the Donaldson--Thomas type invariants and use the functoriality of the constructible function functor $\m F(-)$, we assume that the ground field $\frak K$ is of characteristic zero. We consider MacPherson's Chern class transformation $c_*:\m F(X) \to H_*^{BM}(X)$, which is due to Kennedy \cite{Kennedy}. 

For a morphism $h:V \to X$ and for a constructible function $\delta_X \in \m F(X)$ on the target space $X$, we have
$$\int_V c_*(h^*\delta_X) = \chi(V,h^*\delta_X) = \chi^{\delta_X}(V \xrightarrow h X)^V,$$
$$\int_X c_*(h_*h^*\delta_X) = \chi(X,h_*h^*\delta_X) = \chi^{\delta_X}(V \xrightarrow h X)_X.$$
Here $c_*(h^*\delta_X) \in H_*^{BM}(V)$ on the side of the source space $V$ and $c_*(h_*h^*\delta_X) \in H_*^{BM}(X)$ on the side of the target space $X$. Hence when we want to deal with them as the homomorphism from $K_0(\m V/X)$ to $H_*^{BM}(X)$, we should consider the higher analogues $c_*(h_*h^*\delta_X)$, which we denote by
$$c_*^{\delta_X}(V \xrightarrow h X) := c_*(h_*h^*\delta_X) \in H_*^{BM}(X).$$
On the other hand we denote
$$\overline{c_*^{\delta_X}}(V \xrightarrow h X) := c_*(h^*\delta_X) \in H_*^{BM}(V) .$$
Note that 
\begin{itemize}
\item $c_*^{\delta_X}(V \xrightarrow h X)  = h_*(\overline{c_*^{\delta_X}}(V \xrightarrow h X))$,
\item for an isomorphism $id_X: X \to X$, these two classes are identical and denoted simply by
$c_*^{\delta_X}(X):= c_*(\delta_X) = c_*^{\delta_X}(X \xrightarrow {id_X} X) = \overline{c_*^{\delta_X}}(X \xrightarrow {id_X} X).$ \\
\end{itemize}

In the following sections we treat these two objects $c_*^{\delta_X}(V \xrightarrow h X)$ and $\overline{c_*^{\delta_X}}(V \xrightarrow h X)$ separately, since they have different natures.

\section{Motivic Aluffi-type classes}

For the twisted Behrend function $\widetilde \nu_X$ the Chern class
$c_*^{\widetilde \nu_X}(X)$ is called the Aluffi class and denoted by $c_*^{A\ell}(X)$ (cf. \cite{Aluffi}). Note that
$\int_X c_*^{A\ell}(X) = (-1)^{\op {dim} X} \chi^{DT}(X)$. In \cite{Behrend} the untwisted one $c_*^{\nu_X}(X)$ is called the Aluffi class, in which case $\int_X c_*^{\nu_X}(X) = \chi^{DT}(X)$. But for the sake of later presentation, we stick to the twisted one. In this sense, the Chern class $c_*^{\delta_X}(V \xrightarrow h X)$ defined above is called a \emph{generalized Aluffi class of a morphism $h:V \to X$ associated to a constructible function $\delta_X \in \m F(X)$}. So the original Aluffi class is $c_*^{\widetilde \nu_X}(X \xrightarrow {\op{id}_X} X)$.

\begin{lem}The following formulae hold:
\begin{enumerate}
\item If $V \xrightarrow {h} X \cong V' \xrightarrow {h'} X$, i.e., there exists an isomorphism $k:V \xrightarrow {\cong} V'$ such that $h= h'\circ k$, then we have $c_*^{\delta_X}(V \xrightarrow {h} X)= c_*^{\delta_X}(V' \xrightarrow {h'} X)$.
\item For a closed subvariety $W \subset V$,
$$c_*^{\delta_X}(V \xrightarrow {h} X) = c_*^{\delta_X}(W \xrightarrow {h|_W} X) + c_*^{\delta_X}(V \setminus W \xrightarrow {h|_{V \setminus W}} X).$$
\item For morphisms $h_i: V_i \to X_i \, (i=1, 2)$,
$$c_*^{\delta_{X_1} \times \delta_{X_2}}(V_1 \times V_2  \xrightarrow {h_1 \times h_2} X_1 \times X_2)= c_*^{\delta_{X_1}}(V_1  \xrightarrow {h_1} X_1 ) \times c_*^{\delta_{X_2}}(V_2  \xrightarrow {h_2} X_2)_.$$ 
\item $c_*^{\delta_{pt}}(pt \to pt)= \delta_{pt}(pt) \in \bZ.$ 
\end{enumerate}
\end{lem}

\begin{cor}\label {corollary} Let $\delta_X \in \m F(X)$ be a constructible function. Then the following hold:
\begin{enumerate}
\item The map $c_*^{\delta_X}: K_0(\m V/X) \to H_*^{BM}(X)$ defined by
$$c_*^{\delta_X}([V \xrightarrow h X]) := c_*^{\delta_X}(V \xrightarrow h X) = c_*(h_*h^*\delta_X)$$
and linearly extended is a well-defined homomorphism.
\item $c_*^{\delta_X}$ commutes with the exterior product, i.e. for constructible functions $\delta_{X_i} \in \m F(X_i)$ and for $\alp_i \in K_0(\m V/X_i)$,
$$c_*^{\delta_{X_1} \times \delta_{X_2}}(\alp_1 \times \alp_2) = c_*^{\delta_{X_1}}(\alp_1) \times c_*^{\delta_{X_2}}(\alp_2).$$
\end{enumerate}
\end{cor}

\begin{rem} If $\delta_X$ is some function well-defined on $X$ such as the characteristic function $\jeden_X$, the Behrend function $\nu_X$, the twisted Behrend function $\widetilde \nu_X$, and if it is multiplicative, i.e. $\delta_{X \times Y} = \delta_X \times \delta_Y$, then the above Corollary \ref{corollary} (2) can be simply rewritten as
$c_*^{\delta_{X_1 \times X_2}}(\alp_1 \times \alp_2) = c_*^{\delta_{X_1}}(\alp_1) \times c_*^{\delta_{X_2}}(\alp_2).$
\end{rem}

\begin{rem} If $X$ is smooth, then we have
$c_*^{A\ell}([V \xrightarrow h X]) = c_*(h_*h^*\nu_X) = h_*c_*(h^*\jeden_X) = h_*c_*(\jeden_V) =  h_*c_*^{SM}(V)$ is the pushforward of the Chern--Schwartz--MacPherson class of $V$, thus it depends on the morphism $h:V \to X$, although the degree zero part of it, i.e. the twisted Donaldson--Thomas type invariant is nothing but the Euler characteristic of $V$, thus it does not depend on the morphism at all. Therefore the higher class version is more subtle. 
\end{rem}

The part $h_*h^*\delta_X$ can be formulated as follows. Given a constructible function $\delta_X \in \m F(X)$, we define
$$[\delta_X]: K_0(\m V/X) \to \m F(X)$$
by $[\delta_X]([V \xrightarrow {h} X]) := h_*h^*\delta_X$ and extend it linearly. Note that $\jeden_*: K_0(\m V/X) \to \m F(X)$ is nothing but $[\jeden_X]: K_0(\m V/X) \to \m F(X)$. It is straightforward to see the following.

\begin{lem}\label{lemma4.5} For \emph{any} morphism $g:X \to Y$ and \emph{any} constructible function $\delta_{Y} \in \m F(Y)$, the following diagrams commute:
$$\begin{CD}
 K_0(\m V/X) @> [g^*\delta_{Y}] >> \m F(X) \\
@V g_* VV  @VV  g_*V \\
 K_0(\m V/Y) @>> [\delta_{Y}]  >  \m F(Y).
\end{CD} , \qquad  
\begin{CD}
 K_0(\m V/Y) @> [\delta_{Y}] >> \m F(Y) \\
@V g^* VV  @VV  g^*V \\
 K_0(\m V/X) @>> [g^*\delta_{Y}]  > \m F(X).
\end{CD} 
$$
\end{lem}

The following corollary follows from MacPherson's theorem \cite{Mac1} and our previous results \cite{Sch-VRR, Yokura-VRR}, and here we need the properness of the morphism $g: X \to Y$, since we deal with the pushforward homomorphism for the Borel--Moore homology. 
$c_*^{\delta_X}: K_0(\m V/X) \to H_*^{BM}(X)$ is the composite of $[\delta_X]: K_0(\m V/X) \to \m F(X)$ and MacPherson's Chern class $c_*$, in particular $c_*^{A\ell}:K_0(\m V/X) \to H_*^{BM}(X)$ is $c_*^{A\ell}=c_*\circ [\widetilde{\nu_X}].$ Hence we have the following corollary:
\begin{cor}\label{corollary4.6} 
\begin{enumerate}
\item For a proper morphism $g:X \to Y$ and \emph{any} constructible function $\delta_{Y} \in \m F(Y)$, the following diagram commutes:
$$\begin{CD}
 K_0(\m V/X) @> c_*^{g^*\delta_{Y}} >> H_*^{BM}(X)\\
@V g_* VV   @VV  g_*V \\
 K_0(\m V/Y)  @>> c_*^{\delta_{Y}} > H_*^{BM}(Y).
\end{CD} 
$$
\item For a smooth morphism $g:X \to Y$ with $c(T_g)$ being the total Chern cohomology class of the relative tangent bundle $T_g$ of the smooth morphism and $g^*: H_*^{BM}(Y) \to H_*^{BM}(X)$ the Gysin homomorphism (\cite[Example 19.2.1]{Fulton-Book}) , the following diagram commutes:
$$\begin{CD}
 K_0(\m V/Y) @> c_*^{\delta_{Y}}  >> H_*^{BM}(Y)\\
@V g^* VV  @VV c(T_g) \cap g^*V \\
 K_0(\m V/X) @>>  c_*^{g^*\delta_{Y}}  > H_*^{BM}(X).
\end{CD} 
$$
\end{enumerate}
\end{cor}

In particular we get the following theorem for the Aluffi class $c_*^{A\ell}: K_0(\m V/-) \to H_*^{BM}(-)$:

\begin{thm} For a smooth proper morphism $g:X \to Y$  the following diagrams commute:
$$\begin{CD}
 K_0(\m V/X) @> c_*^{A\ell} >> H_*^{BM}(X)\\
@V g_* VV   @VV  g_*V \\
 K_0(\m V/Y)  @>> c_*^{A\ell} > H_*^{BM}(Y), 
\end{CD} \quad
\begin{CD}
 K_0(\m V/Y) @> c_*^{A\ell}  >> H_*^{BM}(Y)\\
@V g^* VV  @VV c(T_g) \cap g^*V \\
 K_0(\m V/X) @>>  c_*^{A\ell}  > H_*^{BM}(X).
\end{CD} 
$$
They are respectively Grothendieck--Riemann--Roch type and a Verdier--Riemann--Roch type formulas.
\end{thm}


\begin{rem}
 In the above theorem the smoothness of the morphism $g:X \to Y$ is crucial and the Aluffi class homomorphism $c_*^{Al}: K_0(\m V/X) \to H_*^{BM}(X)$ cannot be captured as a natural transformation in a full generality, i.e. natural for any morphism. Indeed, if it were the case, then $c_*^{Al}: K_0(\m V/-) \to H_*^{BM}(-) \hookrightarrow H_*^{BM}(-)\otimes \bQ$ becomes a natural transformation such that for any smooth variety $Y$ we have 
$$c_*^{A\ell}([X \xrightarrow {\op{id}_X} X]) = c(T_X) \cap [X].$$
Let ${T_y}_*: K_0(\m V/-) \to H_*^{BM}(-)\otimes \bQ[y]$ be the motivic Hirzebruch class transformation \cite{BSY}. Then it follows from \cite{BSY} that $c_*^{A\ell} = {T_{-1}}_*:K_0(\m V/-) \to H_*^{BM}(-) \hookrightarrow H_*^{BM}(-)\otimes \bQ$, thus for any variety $X$, singular or non-singular, we have
 $$c_*^{A\ell} ([X \xrightarrow {\op{id}_X} X]) = c_*^{SM}(X) = c_*(\jeden_X)$$
In particular $\int_X c_*(\jeden_X) = \chi(X)$ the topological Euler--Poincar\'e characteristic, which is a contradiction to the fact that 
 $$\int_X c_*^{A\ell} ([X \xrightarrow {\op{id}_X} X]) = (-1)^{\op{dim}X}\chi^{DT}(X).$$
\end{rem}
\begin{rem}
 In fact $c_*^{\jeden_X}$ is equal to the  motivic Chern class transformation ${T_{-1}}_*:K_0(\m V/X) \to H_*^{BM}(X) \hookrightarrow H_*^{BM}(X)\otimes \bQ.$\\\\
 \end{rem}

$K_0(\m V/X)$ is certainly a ring with the following fiber product
$$[V \xrightarrow h X]\cdot [W \xrightarrow k X] := [V \times _X W \xrightarrow {h \times_X k} X].$$
\begin{pro}
The operation $h_*h^*\delta_X$ of pullback followed by pushforward of a constructible function makes $\m F(X)$ a $K_0(\m V/X)$-module with the product $[V \xrightarrow h X]\cdot \delta_X:= h_*h^*\delta_X$. Namely, the following properties hold:
\begin{itemize}
\item $([V \xrightarrow h X] + [W \xrightarrow k X])\cdot \delta_X = [V \xrightarrow h X] \cdot \delta_X + [W \xrightarrow k X]\cdot \delta_X.$
\item  $([V \xrightarrow h X]\cdot [W \xrightarrow k X])\cdot \delta_X =  [V \xrightarrow h X]\cdot ([W \xrightarrow k X]\cdot \delta_X).$ 
\item $[V \xrightarrow h X] \cdot (\delta'_X + \delta''_X) = [V \xrightarrow h X] \cdot \delta'_X + [V \xrightarrow h X] \cdot\delta''_X.$
\end{itemize} 
\end{pro}

Then the operation $h_*h^*\delta_X$ gives rise to a map $\Phi: K_0(\m V/X) \otimes \m F(X) \to \m F(X)$ and the composition $\Phi c_* := c_* \circ \Phi: K_0(\m V/X) \otimes \m F(X) \to H^{BM}_*(X)$ of $\Phi$ and MacPherson's Chern class transformation $c_*$ is a kind of extension of $c_*$. 

\begin{lem} For any morphism $g: X \to Y$ the following diagram commutes:
$$\begin{CD}
 K_0(\m V/Y) \otimes \m F(Y) @> \Phi >> \m F(Y)\\
@V g^* \otimes g^* VV  @VV g^*V \\
 K_0(\m V/X) \otimes \m F(X) @>>  \Phi > \m F(X).
\end{CD} 
$$
\end{lem}

\begin{cor} For a smooth morphism $g:X \to Y$ the following diagram commutes:
$$\begin{CD}
 K_0(\m V/Y) \otimes \m F(Y) @> \Phi c_* >> H^{BM}_*(Y)\\
@V g^* \otimes g^* VV  @VV c(T_g) \cap g^*V \\
 K_0(\m V/X) \otimes \m F(X) @>>  \Phi c_* > H^{BM}_*(X).
\end{CD} 
$$
\end{cor}
\begin{rem} Fix $\delta_Y \in \m F(Y)$, the composite of the inclusion homomorphism $i_{\delta_Y}: K_0(\m V/Y) \to K_0(\m V/Y)\otimes \m F(Y)$ defined by $i_{\delta_Y}(\alpha):= \alpha \otimes \delta_Y$ and the map $\Phi:K_0(\m V/Y)\otimes \m F(Y) \to \m F(Y)$ is the homomorphism $[\delta_Y]$; $\Phi \circ i_{\delta_Y}= [\delta_Y]: K_0(\m V/F) \to \m F(Y).$ The right-hand-sided commutative diagram in Lemma \ref{lemma4.5} is the outer square of the following commutative diagrams:
$$\begin{CD}
 K_0(\m V/Y)  @> i_{\delta_Y} >>K_0(\m V/Y) \otimes \m F(Y) @> \Phi  >> \m F(Y)\\
@V g^* VV @VV g^* \otimes g^* V  @VV g^*V \\
 K_0(\m V/X) @>>  i_{g^*\delta_Y}  > K_0(\m V/X) \otimes \m F(X) @>>  \Phi > \m F(X) .
\end{CD} 
$$
Furthermore, if $g:X \to Y$ is smooth, we get the following commutative diagrams:
$$\begin{CD}
 K_0(\m V/Y)  @> i_{\delta_Y} >>K_0(\m V/Y) \otimes \m F(Y) @> \Phi  >> \m F(Y)@> c_*>> H_*^{BM}(Y)\\
@V g^* VV @VV g^* \otimes g^* V  @VV g^*V @VV c(T_g) \cap g^*V \\
 K_0(\m V/X) @>>  i_{g^*\delta_Y}  > K_0(\m V/X) \otimes \m F(X) @>>  \Phi > \m F(X) @>>  c_*> H^{BM}_*(X),
\end{CD} 
$$
the outer square of which is the commutative diagram in Corollary \ref{corollary4.6} (2).

\end{rem}
\begin{rem} As to the pushforward we do knot know if there is a reasonable pushforward $?:K_0(\m V/X) \otimes \m F(X) \to K_0(\m V/Y) \otimes \m F(Y)$ such that the following diagram commutes:
$$\begin{CD}
 K_0(\m V/X) \otimes \m F(X) @> \Phi >> \m F(X)\\
@V ? VV  @VV g_*V \\
 K_0(\m V/Y) \otimes \m F(Y) @>>  \Phi > \m F(Y).
\end{CD} 
$$
At the moment we can see only that the following diagrams commute:
$$\xymatrix{
 K_0(\m V/X)  \ar[d]_{g_*}\ar[r]^{i_{g^*{\delta_Y}}\qquad } & K_0(\m V/X) \otimes \m F(X) \ar[r]^{\qquad \Phi}  & \m F(X) \ar[d]^{g_*}  \ar[r]^{c_*} & H_*^{BM}(X) \ar[d]^{g_*} \\
 K_0(\m V/Y) \ar[r]_{i_{\delta_Y} \qquad }  &  K_0(\m V/Y) \otimes \m F(Y) \ar[r]_{\qquad \Phi} & \m F(Y) \ar[r]_{c_*} & H^{BM}_*(Y)}
 $$
\end{rem}


\section{Naive Motivic Donaldson--Thomas type Hirzebruch classes}

In this section we give a further generalization of the above generalized Aluffi class $c_*^{\delta}(X)$, using the motivic Hirzebruch class transformation ${T_y}_*:K_0(\m V/-) \to H_*^{BM}(-)\otimes \bQ[y]$.

In the above argument, a key part is the operation of \emph{ pullback-followed-by-pushforward $h_*h^*$} for a morphism $h:V \to X$ on a fixed or chosen constructible function $\delta_X$ of the target space $X$.  It is quite natural to do the same operation on $K_0(\m V/X)$ itself. For that purpose we need to define a motivic element $\delta_X^{mot} \in K_0(\m V/X)$ corresponding to the constructible function $\delta_X$; in particular we need to define a reasonable motivic element $\nu_X^{mot}\in K_0(\m V/X)$ corresponding to the Behrend function $\nu_X \in \m F(X)$. 

By considering the isomorphism
$\jeden:\m Z(X) \xrightarrow {\cong} \m F(X), \, \jeden \left(\sum_V n_V [V] \right) := \sum_V n_V \jeden_V$, we define another distinguished integral cycle:
$\frak D_X := {\jeden}^{-1}(\nu_X)  \, \left (= \jeden^{-1} \circ \op{Eu} (\frak C_X)\right ).$
Then we set $\nu_X^{mot}:= [\frak D_X \to X].$ This can be put in as follows. Let $\frak s:\m F(X) \to K_0(\m V/X)$ be the section of $\jeden_*: K_0(\m V/X) \to \m F(X)$ defined by $\frak s(\jeden_S):= [S \hookrightarrow X]$. Then 
$\nu_X^{mot} = \frak s(\nu_X)$. Another way is $\nu_X^{mot}:= \sum_n n[\nu_X^{-1}(n) \hookrightarrow X]$ (see \cite{Davison}).

\begin{rem} Obviously the homomorphism $[\jeden_X] = \jeden_*: K_0(\m V/X) \to \m F(X)$ is not injective and its kernel is infinite.
In the case when $X$ is the critical set of a regular function $f: M \to \bC$, then there is a notion of ``motivic element" (which is called the ``motivic Donaldson--Thomas invariant") corresponding to the Behrend function (which is in this case described via the Milnor fiber), using the motivic Milnor fiber, due to Denef--Loeser. In our general case, we do not have such a sophisticated machinery available, thus it seems to be natural to define a motivic element $\nu_X^{mot}$ naively as above.
\end{rem}

Let $\Psi: K_0(\m V/X) \otimes K_0(\m V/X) \to K_0(\m V/X)$ be the fiber product mentioned before:
$$\Psi \left ([V \xrightarrow {h} X ] \otimes [W \xrightarrow k X] \right):= [V \xrightarrow {h} X ]\cdot [W \xrightarrow k X] = [V \times_X W \xrightarrow {h \times _X k} X].$$
Since $[\delta_X] = \Phi \circ i_{\delta_X}:K_0(\m V/X) \xrightarrow {i_{\delta_X}} K_0(\m V/X)\otimes \m F(X) \xrightarrow {\Phi} \m F(X)$ with $\delta_X \in \m F(X)$, we consider its ``motivic" analogue, which means the following homomorphism
$$[\gamma_X]: K_0(\m V/X) \xrightarrow {i_{\gamma_X}} K_0(\m V/X)\otimes K_0(\m V/X) \xrightarrow {\Psi} K_0(\m V/X),$$
where $\gamma_X \in K_0(\m V/X)$ and $i_{\gamma_X}:K_0(\m V/X) \to K_0(\m V/X)\otimes K_0(\m V/X)$ is defined by $i_{\gamma_X}(\alpha):= \alpha \otimes \gamma_X.$

\begin{pro}\label{triangle} Let $\gamma_X \in K_0(\m V/X)$. Then the following diagram commutes:
$$\xymatrix{
 K_0(\m V/X) \ar[dr]_ {[\jeden_*(\gamma_X)]}\ar[rr]^ {[\gamma_X]} && K_0(\m V/X) \ar[dl]^{\jeden_*}\\
& \m F(X).}$$
\end{pro}
\begin{proof} Let $\ga_X := [S \xrightarrow {h_S} X]$. Then it suffices to show the following
$$\left (\jeden_* \circ \left [[S \xrightarrow {h_S} X] \right ] \right ) ([V \xrightarrow h X]) = \left [\jeden_* \left ([S \xrightarrow {h_S} X] \right) \right ]([V \xrightarrow h X]).$$
This can be proved using the fiber square \,\, 
$\begin{CD}
V \times _X S @> {\widetilde h} >> S\\
@V {\widetilde {h_S}} VV   @VV  h_S V \\
V  @>> h > X. 
\end{CD} 
$ 

\begin{align*}
\left (\jeden_* \circ \left [[S \xrightarrow {h_S} X] \right ] \right ) ([V \xrightarrow h X]) &= \jeden_* \left (\left [[S \xrightarrow {h_S} X] \right ] ([V \xrightarrow h X]) \right ) \\
& = \jeden_* ([V \times_ X S \xrightarrow {h \circ \widetilde {h_S}} X ]) \\
& = (h \circ \widetilde {h_S})_* \jeden_{V \times_ X S} \,\, \, \text {(by the definition of $\jeden_*$)}\\
& = h_* \widetilde {h_S}_* \jeden_{V \times_ X S}\\
& = h_* \widetilde {h_S}_* \widetilde h^* \jeden_S\\
& = h_* h^* (h_S)_* \jeden_S \quad \text {(since $\widetilde {h_S}_* \widetilde h^* =h^* (h_S)_*$)} \\
& = h_* h^* \left (\jeden_* ([S \xrightarrow {h_S} X]) \right )\\
& = \left [\jeden_* \left ([S \xrightarrow {h_S} X] \right) \right ]([V \xrightarrow h X]).
\end{align*}
\end{proof}

\begin{cor} 
\begin{enumerate}
\item Let $\delta_X \in \m F(X)$ and let $\delta^{mot}_X\in K_0(\m V/X)$ be such that $\jeden_*(\delta^{mot}_X) = \delta_X$. Then we have
$$\xymatrix{
 K_0(\m V/X) \ar[dr]_ {[\delta_X ]}\ar[rr]^ {[\gamma_X]} && K_0(\m V/X) \ar[dl]^{\jeden_*}\\
& \m F(X).}$$
The motivic element $\delta^{mot}_X$ is called a naive motivic analogue of $\delta_X$.

\item In particular, we have
$$\xymatrix{
 K_0(\m V/X) \ar[dr]_ {[\nu_X]}\ar[rr]^ {[\nu_X^{mot}]} && K_0(\m V/X) \ar[dl]^{\jeden_*}\\
& \m F(X).}$$
\end{enumerate}
\end{cor}

\begin{rem}Here we emphasize that the following diagrams commutes:
$$\xymatrix{
 K_0(\m V/X) \ar[dr]_ {[\nu_X]}\ar[rr]^ {[\nu_X^{mot}]} && K_0(\m V/X) \ar[dl]^{\jeden_*}\ar[dr]^{{T_{-1}}_*}\\
& \m F(X)\ar[rr]_{c_* \otimes \bQ} && H_*^{BM}(X) \otimes \bQ.}$$
Thus, modulo the torsion and the choices of motivic elements
$\nu^{mot}_X$, the composite ${T_{-1}}_* \circ [\nu^{mot}_X]$
is a higher class analogue of the Donaldson--Thomas type invariant. Thus it would be natural or reasonable to generalize the Donaldson--Thomas type 
invariant using the motivic Hirzebruch class ${T_y}_*$.\\
\end{rem}

Let $\gamma_X \in K_0(\m V/X), \gamma_Y \in K_0(\m V/Y)$. Then for any morphism $g:X \to Y$
the following diagrams commute:

$$
\begin{CD}
 K_0(\m V/X) @> [\gamma_X]>> K_0(\m V/X)\\
@V g_* VV   @VV  g_*V \\
 K_0(\m V/Y)  @>> [g_*\gamma_X] > K_0(\m V/Y), 
\end{CD} \,\,  \text {or} \, 
\begin{CD}
 K_0(\m V/X) @> i_{\gamma_X}>> K_0(\m V/X) \otimes K_0(\m V/X) @> \Psi >>K_0(\m V/X)\\
@V g_* VV   @VV  g_* \otimes g_*V  @VV  g_*V\\
 K_0(\m V/Y)  @>> i_{g_*\gamma_X}> K_0(\m V/Y) \otimes K_0(\m V/Y) @>> \Psi > K_0(\m V/Y)
\end{CD} 
$$
$$
\begin{CD}
 K_0(\m V/X) @> [g^*\gamma_Y]>> K_0(\m V/X)\\
@V g_* VV   @VV  g_*V \\
 K_0(\m V/Y)  @>> [\gamma_Y] > K_0(\m V/Y), 
\end{CD} \,\,  \text {or} \, 
\begin{CD}
 K_0(\m V/X) @> i_{g^*\gamma_Y}>> K_0(\m V/X) \otimes K_0(\m V/X) @> \Psi >>K_0(\m V/X)\\
@V g_* VV   @VV  g_* \otimes g_*V  @VV  g_*V\\
 K_0(\m V/Y)  @>> i_{\gamma_Y}> K_0(\m V/Y) \otimes K_0(\m V/Y) @>> \Psi > K_0(\m V/Y)
\end{CD}
$$
$$
\begin{CD}
 K_0(\m V/Y) @> [\gamma_Y]   >> K_0(\m V/Y)\\
@V g^* VV  @VV g^*V \\
 K_0(\m V/X) @>>  [g^*\gamma_Y]  > K_0(\m V/X).
\end{CD} \,\,  \text {or} \, 
\begin{CD}
 K_0(\m V/Y) @> i_{\gamma_Y}>> K_0(\m V/Y) \otimes K_0(\m V/Y) @> \Psi >>K_0(\m V/Y)\\
@V g^* VV   @VV  g^* \otimes g^*V  @VV  g^*V\\
 K_0(\m V/Y)  @>> i_{g^*\gamma_Y}> K_0(\m V/X) \otimes K_0(\m V/X) @>> \Psi > K_0(\m V/X)
\end{CD} 
$$
Hence we get the following corollary

\begin{cor}
\begin{enumerate}
\item Let $\gamma_X \in K_0(\m V/X),\gamma_Y \in K_0(\m V/Y)$. For a proper morphism $g: X \to Y$ the following diagrams commute:
$$\begin{CD}
 K_0(\m V/X) @> {T_y}_*\circ \,  [\gamma_X] >> H_*^{BM}(X)\otimes \bQ[y]\\
@V g_* VV @VV  g_*V \\
 K_0(\m V/Y) @>> {T_y}_*\circ \, [g_*\gamma_X]  > H_*^{BM}(Y)\otimes \bQ[y],
\end{CD}  \quad
\begin{CD}
 K_0(\m V/X) @> {T_y}_*\circ \,  [g^*\gamma_Y] >> H_*^{BM}(X)\otimes \bQ[y]\\
@V g_* VV @VV  g_*V \\
 K_0(\m V/Y) @>> {T_y}_*\circ \, [\gamma_Y]  > H_*^{BM}(Y)\otimes \bQ[y],
\end{CD}
$$

\item
For a proper smooth morphism $g:X \to Y$ and for $\gamma_Y \in K_0(\m V/Y)$ the following diagrams are commutative:
$$\begin{CD}
K_0(\m V/Y) @> {T_y}_*\circ \,  [\gamma_Y] >> H_*^{BM}(Y)\otimes \bQ[y]\\
@V g^* VV @VV  td_y(T_g) \cap g^*V \\
 K_0(\m V/X) @>> {T_y}_*\circ \, [g^*\gamma_Y]  > H_*^{BM}(X)\otimes \bQ[y].
\end{CD} 
$$

\item Let $\widetilde {\nu}_X^{mot} := (-1)^{\op {dim} X} \nu_X^{mot}$, the twisted one. Let ${T_y}_*^{DT}:= {T_y}_*\circ [\widetilde {\nu}_X^{mot}]$. For a proper smooth morphism $g:X \to Y$ the following diagrams are commutative:

$$\begin{CD}
 K_0(\m V/X) @> {T_y}_*^{DT}>> H_*^{BM}(X)\otimes \bQ[y]\\
@V g_* VV @VV  g_*V \\
 K_0(\m V/Y) @>> {T_y}_*^{DT}  > H_*^{BM}(Y)\otimes \bQ[y],
\end{CD}  \quad
\begin{CD}
K_0(\m V/Y) @> {T_y}_*^{DT} >> H_*^{BM}(Y)\otimes \bQ[y]\\
@V g^* VV @VV  td_y(T_g) \cap g^*V \\
 K_0(\m V/X) @>> {T_y}_*^{DT} > H_*^{BM}(X)\otimes \bQ[y].
\end{CD} 
$$
\end{enumerate}
\end{cor}

\begin{rem} The commutative diagram in Proposition \ref{triangle} can be described in more details as follows:
 $$\xymatrix{
  K_0(\m V/X) \ar[r]^{i_{\gamma_X}\qquad } & K_0(\m V/X) \otimes K_0(\m V/X) \ar[r]^{\qquad \Psi} \ar[d]_{id \otimes i_{\jeden_X}}  & K_0(\m V/X) \ar[d]^{i_{\jeden_X}} \\
& K_0(\m V/X) \otimes K_0(\m V/X) \otimes \m F(X) \ar[r]^{\qquad \Psi \otimes id} \ar[d]_{id \otimes \Phi}  & K_0(\m V/X)  \otimes \m F(X) \ar[d]^{\Phi} \\
& K_0(\m V/X)\otimes \m F(X)  \ar[r]_{\qquad \Phi }  & \m F(X)}
 $$
If we denote $\Phi(\alpha\otimes \delta_X)$ simply by $\alpha \cdot \delta_X$, then the bottom square on the right-hand-side commutative diagrams means that $(\alpha \cdot \beta )\cdot \delta_X = \alpha \cdot (\beta \cdot \delta_X)$, i.e. the associativity. 
\end{rem} 

\begin{rem} We remark that the following diagrams commute:
\begin{enumerate}
\item for a proper marphism $g:X \to Y$
$$\begin{CD}
\overbrace{K_0(\m V/X) \otimes \cdots \otimes K_0(\m V/X)}^n @> \Psi^{n-1}  >> K_0(\m V/X) @> {T_y}_*>> H_*^{BM}(X)\otimes \bQ[y]\\
@VV g_* \otimes \cdots \otimes g_* V  @VV g_*V @VV g_*V \\
\underbrace{K_0(\m V/Y) \otimes \cdots \otimes K_0(\m V/Y)}_n @>>  \Psi^{n-1} > K_0(\m V/Y) @>> {T_y}_*> H_*^{BM}(Y)\otimes \bQ[y],
\end{CD} 
$$
\item for a proper smooth morphism $g:X \to Y$
$$\begin{CD}
\overbrace{K_0(\m V/Y) \otimes \cdots \otimes K_0(\m V/Y)}^n @> \Psi^{n-1}  >> K_0(\m V/X) @> {T_y}_*>> H_*^{BM}(X)\otimes \bQ[y]\\
@VV g^* \otimes \cdots \otimes g^* V  @VV g^*V @VV c(T_g) \cap g_*V \\
\underbrace{K_0(\m V/X) \otimes \cdots \otimes K_0(\m V/X)}_n @>>  \Psi^{n-1} > K_0(\m V/X) @>> {T_y}_*> H_*^{BM}(X)\otimes \bQ[y],
\end{CD} 
$$
\end{enumerate}
Here $\Psi^{n-1}([V \to X]):= [V \to X] \cdot \, \, \cdots \, \, \cdot [V \to X]$ is the fiber product of $n$ copies of $[V \to X]$. When $n=1$, $\Psi^0:= \op{id}_{K_0(\m V/X)}$ is understood to be the identity. Let $P(t) := \sum a_i t^i \in \bQ[t]$ be a polynomial. Then we define the polynomial transformation $\Psi_{P(t)}: K_0(\m V/X) \to K_0(\m V/X)$ by
$$\Psi_{P(t)} ([V \xrightarrow h X]):= \sum a_i \Psi^{i-1}([V \to X]).$$
Then we have the following commutative diagrams.
\begin{enumerate}
\item for a proper morphism $g:X \to Y$
$$\begin{CD}
K_0(\m V/X) @> \Psi_{P(t)}  >> K_0(\m V/X) @> {T_y}_*>> H_*^{BM}(X)\otimes \bQ[y]\\
@VV g_* V  @VV g_*V @VV g_*V \\
K_0(\m V/Y) @>> \Psi_{P(t)} > K_0(\m V/Y) @>> {T_y}_*> H_*^{BM}(Y)\otimes \bQ[y],
\end{CD} 
$$
\item for a proper smooth morphism $g:X \to Y$
$$\begin{CD}
K_0(\m V/Y) @> \Psi_{P(t)} >> K_0(\m V/X) @> {T_y}_*>> H_*^{BM}(X)\otimes \bQ[y]\\
@VV g^* V  @VV g^*V @VV c(T_g) \cap g_*V \\
K_0(\m V/X) @>> \Psi_{P(t)} > K_0(\m V/X) @>> {T_y}_*> H_*^{BM}(X)\otimes \bQ[y],
\end{CD} 
$$
\end{enumerate}
These are a ``motivic" analogue of the corresponding case of constructible functions:
\begin{enumerate}
\item for a proper morphism $g:X \to Y$
$$\begin{CD}
\m F(X) @> \m F_{P(t)}  >> \m F(X) @> c_*>> H_*^{BM}(X)\\
@VV g_* V  @VV g_*V @VV g_*V \\
\m F(Y) @>> \m F_{P(t)} > \m F(Y) @>> c_*> H_*^{BM}(Y)
\end{CD}
$$
\item for a proper smooth morphism $g:X \to Y$
$$\begin{CD}
\m F(Y) @> \m F_{P(t)} >> \m F(Y) @> c_*>> H_*^{BM}(Y)\\
@VV g^* V  @VV g^*V @VV c(T_g) \cap g^*V \\
\m F(X) @>> \m F_{P(t)} > \m F(X) @>> c_*> H_*^{BM}(X)
\end{CD}
$$
\end{enumerate}
Here $\m F_{P(t)}(\beta) := \sum a_i \beta^i$. Note also that the following diagram commutes
$$\begin{CD}
K_0(\m V/X) @> \Psi_{P(t)} >> K_0(\m V/X)\\
@VV \jeden_* V  @VV \jeden_*V  \\
\m F(X) @>> \m F_{P(t)} > \m F(X).
\end{CD}
$$

\end{rem}

\begin{defn} 
\begin{enumerate}
\item We refer to the following class
$${T_y}_*^{DT}(X) := \left({T_y}_*^{DT} \right) ([X \xrightarrow {id_X} X]) = {T_y}_*([\widetilde \nu_X^{mot}])$$
as the \emph{naive motivic Donaldson--Thomas type Hirzebruch class} of $X$.
\item The degree zero of the naive motivic Donaldson--Thomas type Hirzebruch class is called the \emph{naive motivic Donaldson--Thomas type $\chi_y$-genus} of $X$:
$$\chi_y^{DT}(X) := \int_X {T_y}_*^{DT}(X).$$
\end{enumerate}
\end{defn}

\begin{rem}
The cases of the three special values $y = -1, 0, 1$ are the following.
\begin{enumerate}
\item For $y=-1$, ${T_{-1}}_*^{DT}(X) = {T_{-1}}_*([\widetilde \nu_X^{mot}])=  c_*^{A\ell}(X)$.
\item For $y=0$, ${T_0}_*^{DT}(X) ={T_{0}}_*([\widetilde \nu_X^{mot}])=:td_*^{A\ell}(X)$, called an ``Aluffi-type" Todd class of $X$.
\item For $y=1$, ${T_1}_*^{DT}(X) ={T_{1}}_*([\widetilde \nu_X^{mot}])=:L_*^{A\ell}(X)$, called an ``Aluffi-type" Cappell--Shaneson L-homology class of $X$.
\end{enumerate}

The degree zero part of these three motivic classes are respectively: 
 \begin{enumerate}
\item for $y=-1$, $\chi_{-1}^{DT}(X) = (-1)^{\op{dim} X} \chi^{DT}(X)$, the original Donaldson--Thomas type invariant (i.e. Euler characteristic) of $X$ with the sign; 
\item for $y=0$, $\chi_{0}^{DT}(X) =:\chi_a^{DT}(X)$, called a \emph{naive Donaldson--Thomas type arithmetic geneus} of $X$ and 
\item for $y=1$, $\chi_{-1}^{DT}(X) = \sigma^{DT}(X)$
, called a \emph{naive Donaldson--Thomas type signature} of $X$.
\end{enumerate}
\end{rem}

\begin{rem} Since $\widetilde {\nu}_X(x) = 1$ for a smooth point $x \in X$, we have that $\widetilde{\nu}_X = \jeden_ X + \alpha_{X_{sing}}$ for some constructiblee functions $\alpha_{X_{sing}}$ supported on the singular locus $X_{sing}$. For example, consider the simplest case that $X$ has one isolated singularity $x_0$, say $\widetilde {\nu}_X = \jeden_X + a_0 \jeden_{x_0}$. Then
$$\widetilde {\nu}_X^{mot} = [X \xrightarrow {id_X} X] + a_0[x_0 \xrightarrow {i_{x_0}}  X] \in K_0(\m V/X).$$
Here $x_0 \xrightarrow {i_{x_0}}  X$ is the inclusion. Hence we have
\begin{align*}
{T_y}_*^{DT}(X) & = {T_y}_*(\widetilde {\nu}_X^{mot}) \\
& =  {T_y}_*([X \xrightarrow {id_X} X] + a_0[x_0 \xrightarrow {i_{x_0}} X]) \\
& =  {T_y}_*(X) + a_0 (i_{x_0})_*{T_y}_*(x_0)  \\
& =  {T_y}_*(X) + a_0.
\end{align*}
Thus the difference between the motivic DT type Hirzebruch class ${T_y}_*^{DT}(X)$ and the motivic Hirzebruch class 
${T_y}_*(X)$ is just $a_0$, independent of the parameter $y$. Of course, if $\op {dim} X_{sing} \geq 1$, then the difference DOES depend on the parameter $y$. For example, for the sake of simplicity, assume that $\widetilde{\nu}_X = \jeden_ X + a \jeden_{X_{sing}}$. Then the difference is
$${T_y}_*^{DT}(X) - {T_y}_*(X) = a (i_{X_{sing}})_*{T_y}_*(X_{sing}),$$
which certainly depends  on the parameter $y$, \emph{at least for the degree zero part $\chi_y (X_{sing})$}.

If we take a different motivic element $\overline{\nu}_X^{mot} = [X \xrightarrow {id_X} X] + [V \xrightarrow h X]$ such that $\jeden_*([V \xrightarrow h X]) = a_0\jeden_{x_0}$ and $\op {dim}V \geq 1$, then the difference ${T_y}_*^{DT}(X) - {T_y}_*(X) = h_*({T_y}_*(V))$, thus it DOES depend on the parameter $y$, at least for the degree zero part, again.

 In the case when $X$ is the critical locus of a regular function $f: M \to \bC$, the motivic DT invariant $\nu_X^{motivic}$ which DT-theory people consider, using the motivic Milnor fiber, is the latter case, simply due to the important fact that the Behrend function can be expressed using the Milnor fiber. For example, as done in \cite{CMSS}, even for an isolated singularity $x_0$, the difference ${T_y}_*^{DT}(X) - {T_y}_*(X)$ is, up to sign, the $\chi_y$-genus of (the Hodge structure of) the Milnor fiber at the singularity $x_0$, so does depend on the parameter $y$.

So, as long as the Behrend function has some geometric or topological descriptions, e.g., such as Milnor fibers, then one could think of the corresponding motivic elements in a naive or canonical way.

 We will hope to come back to properties of these two classes $td_*^{A\ell}(X)$, $L_*^{A\ell}(X)$ and $\chi_a^{DT}(X)$, $\sigma^{DT}(X)$ and discussion on some relations with other invariants of singularities. \\
  \end{rem}
 
 \begin{rem} In \cite{CMSS} Cappell et al. use the Hirzebruch class transformation 
 $${MHMT_y}_*: K_0(MHM(X)) \to H_*^{BM}(X) \otimes \bQ[y, y^{-1}]$$
 from the Grothendieck group $K_0(MHM(X))$ of the category of mixed Hodge modules (introduced by Morihiko Saito), instead of the Grothendieck group $K_0(\m V/X)$. We could do the same things on ${MHMT_y}_*: K_0(MHM(X)) \to H_*^{BM}(X) \otimes \bQ[y, y^{-1}]$ and get $MHM$-theoretic analogues of the above. We hope to get back to this calculation.
 \end{rem}
 \begin{rem} In \cite{GS} G\"ottsche and Shende made an application of the motivic Hirzebruch class ${T_y}_*$.
 \end{rem}
 \begin{rem} In a successive paper, we intend to apply the motivic Hirzebruch transformation to the motivic vanishing cycle constructed on the Donaldson--Thomas moduli space and announced in \cite{BBDJS, BJM}. This will hopefully provide the ``right" motivic Donaldson--Thomas type Hirzebruch class. 
 


 \end{rem}
\section{A bivariant group of pullbacks of constructible functions and a bivariant-theoretic problem}
In the above section we mainly dealt with the class $c_*^{\delta_X}(V \xrightarrow {h} X)$ of $h:V \to X$ supported on the target space $X$. In this section we deal with the class $\overline {c_*^{\delta_X}}(V \xrightarrow {h} X)$ of $h:V \to X$ supported on the source space $V$. 

The class $c_*^{\delta_X}(V \xrightarrow {h} X)$ is by definition $c_*(h_*h^*\delta_X) = h_*c_*(h^*\delta_X) \in H_*^{BM}(X)$, and can be captured as the image of the homomorphism from two abelian groups assigned to the space $X$. However, when it comes to the case of $\overline {c_*^{\delta_X}}(V \xrightarrow {h} X) \in H_*^{BM}(V)$, one cannot do it. So we approach this class from a bivariant-theoretic viewpoint as follows.

For a morphism $f:X \to Y$ and a constructible function $\delta_Y \in \m F(Y)$, we define $\bF^{\delta_Y}(X \xrightarrow {f} Y)$ as follows:

$$\bF^{\delta_Y}(X \xrightarrow {f} Y):= \left \{ \sum_S a_S {i_S}_* i_S^* f^* \delta_Y \, \Bigl | \, S \, \, \text {are closed subvarieties of } X, a_S \in \bZ \right \},$$
where  $i_S: S \to X$ is the inclusion map. For the sake of simplicity, unless some confusion is possible, we simply denote ${i_S}_* (i_S)^* f^* \delta_Y $ by $(f|_S)^*\delta_Y (= (i_S)^* f^* \delta_Y ) $. In particular, let us consider the twisted Behrend function $\widetilde{\nu}_Y$ as $\delta_Y$, i.e., $\bF^{\widetilde{\nu}_Y}(X \xrightarrow {f} Y)$, which shall be denoted by $\bF^{Beh}(X \xrightarrow f Y)$. It is easy to see the following lemma.

\begin{lem}
\begin{enumerate}
\item If $ Y$ is smooth, then $\bF^{Beh}(X \xrightarrow {f} Y) = \m F(X)$.
\item If $Y$ is singular and $f(X) \cap Y_{sing} = \emptyset$, $\bF^{Beh}(X \xrightarrow {f} Y) = \m F(X)$.
\item If $Y$ is singular and $f(X) \cap Y_{sing} \not = \emptyset$, $\bF^{Beh}(X \xrightarrow {f} Y) \subsetneqq \m F(X)$.
\item $\bF^{Beh}(X \xrightarrow {\pi } pt) = \m F(X)$.
\item If $X$ is smooth, $\bF^{Beh}(X \xrightarrow {\op{id}_X} X) = \m F(X)$.
\item If $X$ is singular, then $\bF^{Beh}(X \xrightarrow {\op{id}_X} X) \subsetneqq \m F(X)$ and in particular, the characteristic function $\jeden_X \not \in \bF^{Beh}(X \xrightarrow {\op{id}_X} X)$.
\end{enumerate}
\end{lem}

In order to show that $\bF^{Beh}(X \xrightarrow {f} Y)$ is a bivariant theory in the sense of Fulton and MacPherson \cite{FM}, first we quickly recall some basics about Fulton--MacPherson's bivariant theory.

\begin{defn}
A {\it bivariant theory} $\bB$ on a category $\m C$ is an assignment to each morphism 
$$X \xrightarrow f Y$$
 in the category $\m C$
a (graded) abelian group
$$\bB(X \xrightarrow f Y),$$ 
which is equipped with the following three basic operations:
\begin{itemize}
\item[(i)]  for morphisms $X \xrightarrow f  Y$ and $Y \xrightarrow g Z$, the {\it product operation} 
$$
\bullet : \bB(X \xrightarrow f  Y) \otimes \bB(Y \xrightarrow g Z) \to \bB(X \xrightarrow {gf} Z)
$$
is defined;
\item[(ii)]  for morphisms $X \xrightarrow f  Y$ and $Y \xrightarrow g Z$ with $f$ proper, the {\it pushforward operation} 
$$
f_*: \bB(X \xrightarrow {gf}  Z) \to \bB(Y \xrightarrow g Z) 
$$
is defined;
\item[(iii)]  for a fiber square 
\quad 
$\CD
X' @> g' >> X \\
@V f' VV @VV f V\\
Y' @>> g > Y, \endCD
$ \quad
the {\it pullback operation} 
$$
g^*: \bB(X \xrightarrow {f}  Y) \to \bB(X' \xrightarrow {f'} Y') 
$$
is defined.
\end{itemize}

These three operations are required to satisfy the seven axioms which are natural properties to make them compatible
each other:
\begin{itemize}
\item[{\bf(B1)}] product is associative;
\item[{\bf(B2)}] pushforward is functorial;
\item[{\bf(B3)}] pullback is functorial;
\item[{\bf(B4)}] product and pushforward commute;
\item[{\bf(B5)}] product and pullback commute;
\item[{\bf(B6)}] pushforward and pullback commute;
\item[{\bf(B7)}] projection formula.
\end{itemize}
\label{bivth}
\end{defn}

\begin{defn}
Let $\bB$ and $\bB'$ be two bivariant theories on a category $\m C.$ Then a {\it Grothendieck transformation}
from $\bB$ to $\bB'$ $$\gamma: \bB \longrightarrow\bB'$$ is a collection of morphisms 
$$\bB(X \xrightarrow f Y) \to \bB'(X \xrightarrow f  Y) $$
for each morphism $X \xrightarrow f Y$ in the category $\m C,$ which preserves the above three basic operations. 
\label{bivth2}
\end{defn}

As to the constructible functions we recall the following fact from \cite{Yokura}:
\begin{thm} If we define $\bF (X \xrightarrow {f} Y):=F(X)$ (ignoring the morphism $f$), then it become a bivariant theory, called the ``simple" bivariant theory of constructible functions with the following three bivariant operations:

\begin{itemize}
\item (bivariant product) 
$$\bullet : \bF(X \xrightarrow {f} Y) \otimes \bF (Y \xrightarrow {g} Z) \to \bF(X \xrightarrow {gf} Z),$$
$$ \alp \bullet \be := \alp \cdot f^* \be.$$

\item (bivariant pushforward) For morphisms $f: X \to Y$
and $g: Y \to Z$ with $f$ \emph {proper}
$$f_{\bigstar}: \bF(X \xrightarrow {gf} Z) \to \bF(Y \xrightarrow {g} Z)$$
$$f_{\bigstar}\alp := f_*\alp.$$

\item (bivariant pullback) For a fiber square \quad 
$\CD
X' @> g' >> X \\
@V f' VV @VV f V\\
Y' @>> g > Y, \endCD
$ \quad 
 $$g^{\bigstar}: \bF(X \xrightarrow {f} Y) \to \bF (X' \xrightarrow {f'} Y')$$
$$g^{\bigstar}\alp := (g')^*\alp.$$
\end{itemize}
\end{thm}

\begin{thm} Here we consider the category of complex algebraic varieties. Then the above group $\bF^{Beh}(X \xrightarrow {f} Y)$ becomes a bivariant theory as a subgroup or subtheory of the above simple bivariant theory $\bF(X \xrightarrow {f} Y)$, provided that we consider smooth morphisms $g$ for the bivariant pullback.

\end{thm}
\begin{proof} All we have to do is to show that those three bivariant operations are well-defined or stable on the subgroup $\bF^{Beh}(X \xrightarrow {f} Y)$. Below, as to bivariant product and bivariant pushforward, we do not need the requirement that $\delta_Y$ is the Behrend function $\nu_Y$, but we need  it for bivariant pullback.

\begin{enumerate}
\item (bivariant product) It suffices to show that
$$(f|_S)^*\delta_Y \bullet (g|_W)^*\delta_Z = (f|_S)^*\delta_Y \cdot f^*(g|_W)^*\delta_Z\in \bF^{\delta_Z}(X \xrightarrow {gf} Z).$$
Since $(f|_S)^*\delta_Y$ is a constructible function on $S$, $(f|_S)^*\delta_Y = \sum_V a_V\jeden_V$ where $V$'s are subvarieties of $S$, hence subvarieties of $X$. Thus
we get
\begin{align*}
(f|_S)^*\delta_Y \cdot f^*(g|_W)^*\delta_Z & = \sum_V a_V \jeden_V \cdot (gf|_{f^{-1}(W)})^*\delta_Z \\
& = \sum_V a_V (gf|_{f^{-1}(W) \cap V})^*\delta_Z 
\end{align*}
Since $f^{-1}(W) \cap V$ is a finite union of subvarieties, it follows that 
$(f|_S)^*\delta_Y \cdot f^*(g|_W)^*\delta_Z \in \bF^{\delta_Z}(X \xrightarrow{gf} Z).$

\item (bivariant pushforward) It suffices to show that 
$$f_*((gf|_S)^*\delta_Z) \in \bF^{\delta_Z}(Y \xrightarrow g Z).$$
More precisely, $f_*((gf|_S)^*\delta_Z) = f_*(i_S)_*(f|_S)^*g^*\delta_Z) = (f|_S)_*(f|_S)^*g^*\delta_Z.$
Now it follows from Verdier's result \cite[(5.1) Corollaire]{Verdier} that the morphism $f|_S: S \to Y$ is a stratified submersion, more precisely there is a filtration of closed subvarieties $V_1 \subset V_2 \subset \cdots \subset V_m \subset Y$ such that the restriction of $f|_S$ to each strata $V_{i+1} \setminus V_i$, i.e., $(f|_S)^{-1}(V_{i+1} \setminus V_i) \to  V_{i+1} \setminus V_i$ is a fiber bundle. Hence the operation $(f|_S)_*(f|_S)^*$ is the same as the multiplication $(\sum_i a_i \jeden_{V_i})\cdot$ with some integers $a_i$'s, i.e., 
$$(f|_S)_*(f|_S)^*g^*\delta_Z = (\sum_i a_i \jeden_{V_i})\cdot g^*\delta_Z = \sum_i a_i (g|_{V_i})^*\delta_Z \in \bF^{\delta_Z}(Y \xrightarrow g Z).$$

\item (bivariant pullback) Here we show that the following is well-defined
$$g^*: \bF^{\delta_Y}(X \xrightarrow f Y) \to \bF^{g^*\delta_Y}(X' \xrightarrow {f'} Y').$$
Consider the following fiber squares:
$$\CD
S' @> {g''} >> S \\
@V i_{S'} VV @VV i_S V\\
X' @> g' >> X \\
@V f' VV @VV f V\\
Y' @>> g > Y.
\endCD
$$

Indeed, 
\begin{align*}
g^*((f|_S)^*\delta_Y) & = (g')^*((f|_S)^*\delta_Y \quad \text {(by definition)} \\
& = (g')^*((i_S)_*(f|_S)^*\delta_Y \quad \text {(more precisely)} \\
& = (i_{S'})_*(g'')^*(i_S)^*f^*\delta_Y\\
& = (i_{S'})_*(i_{S'})^*(f')^*g^*\delta_Y \in \bF^{g^*\delta_Y}(X' \xrightarrow {f'} Y')\\
\end{align*}
Hence, if we take the twisted Behrend function $\widetilde {\nu}_Y$, for a smooth morphism $g:Y' \to Y$ we have $\widetilde {\nu}_{Y'} = g^*\widetilde{\nu}_Y.$

\end{enumerate}

\end{proof}

\begin{prob}\label{bivariant} Can one define a ``bivariant homology theory" $\widetilde{\bH}(X \to Y)$ such that
\begin{enumerate}
\item $\widetilde{\bH}(X \xrightarrow {f} Y) \subseteqq H_*^{BM}(X)$ for any morphism $f:X \to Y$,
\item $\widetilde{\bH}(X \xrightarrow {} Y) = H_*^{BM}(X)$ for a smooth $Y$,
\item the MacPherson's Chern class 
$$c_*:  \bF^{Beh}(X \xrightarrow {f} Y) \to \widetilde{\bH}(X \xrightarrow {f} Y) $$
defined by $c_*({i_S}_*i_S^* f^*\widetilde{\nu}_Y):= {i_S}_*c_*(i_S^* f^*\widetilde{\nu}_Y) \in H_*^{BM}(X)$ and extended linearly,
 becomes a Grothendieck transformation.
 \item if $Y$ is a point $pt$, then $c_*:  F(X) = \bF^{Beh}(X \xrightarrow {f} pt) \to \widetilde{\bH}(X \xrightarrow {f} pt) = H_*^{BM}(X)$
 is equal to the original MacPherson's Chern class homomorphism.
 \end{enumerate}
 \end{prob}
\begin{rem} One simple-minded construction of such a ``bivariant homology theory" $\widetilde{\bH}(X \to Y)$ could be simply the image of $\bF^{Beh}(X \xrightarrow {f} Y)$ under the MacPherson's Chern class $c_*:\m F(X) \to H_*^{BM}(X)$:
$$\widetilde{\bH}(X \to Y):= c_*(\bF^{Beh}(X \xrightarrow {f} Y)).$$\\
\end{rem} 

Before closing this section, we mention a bivariant-theoretic analogue of the covariant functor of conical Lagrangian cycles. 

In \cite{Kennedy} Kennedy proved that $Ch: F(X) \xrightarrow {\cong} \mathcal L(X)$ is an isomorphism.
In general, suppose we have a correspondence $\mathcal H$ such that
\begin{itemize}
\item $\mathcal H$ assigns an abelian group $\mathcal H(X)$ to a variety $X$
\item there is an isomorphism $\Theta_X:F(X) \xrightarrow {\cong} \mathcal H(X)$.
\end{itemize}
Then, if we define the pushforward $f_*:\mathcal H(X) \to \mathcal H(Y)$ for a map $f:X \to Y$ by
$$f^{\mathcal H}_*:= \mathcal H \circ f^F_* \circ \mathcal H^{-1}: \mathcal H(X) \to \mathcal H(Y)$$
then the correspondence $\mathcal H$ becomes a covariant functor \emph{via the covariant functor F}.
Here $f^F_*:F(X) \to F(Y)$, emphasizing the covariant functor $F$.
Similary, if we define the pullback $f^*:\mathcal H(Y) \to \mathcal H(X)$ by
$$f_{\mathcal H}^*:= \mathcal H \circ f_F^* \circ \mathcal H^{-1}: \mathcal H(Y) \to \mathcal H(X)$$
then the correspondence $\mathcal H$ becomes a contravariant functor \emph{via the contravariant functor F}.
Here $f_F^*:F(Y) \to F(X)$.
Furthermore, if we define
$$\bB \mathcal H(X \xrightarrow f Y):= \mathcal H(X)$$
then we get the simple bivariant-theoretic version of the correspondence $\mathcal H$ as follows:
\begin{itemize}
\item (Bivariant product) $\bullet_{\bB \mathcal H} :\bB \mathcal H(X \xrightarrow f Y) \otimes \bB \mathcal H(Y \xrightarrow g Z) \to \bB \mathcal H(X \xrightarrow {gf} Z)$ is defined by
$$ \alpha \bullet_{\bB \mathcal H} \beta := \mathcal H \Bigl (\mathcal H^{-1}(\alpha) \bullet_{\bF} \mathcal H ^{-1}(\beta) \Bigr ).$$
\item (Bivariant pushforward) $f_*^{\bB \mathcal H}: \bB \mathcal H(X \xrightarrow {gf} Z) \to  \bB \mathcal H(Y \xrightarrow g Z)$ is defined by
$$f_*^{\bB \mathcal H} := \mathcal H \circ f_*^{\bF} \circ \mathcal H^{-1}.$$
\item (Bivariant pullback) $g*_{\bB \mathcal H}: \bB \mathcal H(X \xrightarrow {f} Y) \to  \bB \mathcal H(X' \xrightarrow {f'} Y')$ is defined by
$$g^*_{\bB \mathcal H} := \mathcal H \circ f^*_{\bF} \circ \mathcal H^{-1}.$$
\end{itemize}

Clearly we get the canonical Grothendieck transformation
$$\ga_{\Theta}= \Theta : \bF(X \xrightarrow f Y) \to \bB \mathcal H(X \xrightarrow f Y).$$

If we apply this argument to the conical Lagrangian cycle $\mathcal L(X)$ we get the simple bivariant theory of conical Lagrangian cycles 
$$\bL(X \xrightarrow f Y)$$
and also we get the canonical Grothendieck transformation
$$\ga_{Ch}=Ch : \bF(X \xrightarrow f Y) \to \bL(X \xrightarrow f Y).$$
This simple bivariant theory $\bL(X \xrightarrow f Y)$ can be defined or constructed directly as done in \cite{Bussi}, in which one has to go through many geometric and/or topological ingredients.

The Fulton--MacPherson's bivariant theory $\bF^{FM}(X \xrightarrow f Y)$ is a subgroup (or a subtheory) of the simple bivariant theory $\bF(X \xrightarrow f Y) = F(X)$. Then if we define
$$\bL^{FM}(X \xrightarrow f Y) := \gamma_{Ch}(\bF^{FM}(X \xrightarrow f Y))$$
then we can get a finer bivariant theory of conical Lagrangian cycles, putting aside the problem of how we define or describe such a finer bivariant-theoretic conical Lagrangian cycle; it would be much harder than the case of the simple one $\bL(X \xrightarrow f Y)$ done in \cite{Bussi}.

\section{Some more questions and problems}
\subsection{A categorification of Donaldson--Thomas type invariant of a morphism}
 The cardinality $c(F)$ of a finite set $F$, i.e., the number of elements of $F$, satisfies that 
\begin{enumerate}
\item[(1)] $X \cong X'$ (set-isomorphism) $\Longrightarrow$ $c(X) = c(X')$,
\item[(2)] $c(X) = c(Y) + c(X \setminus Y)$ for a subset $Y \subset X$ (a \emph{scissor relation}),
\item[(3)]  $c(X \times Y) = c(X) \times c(Y)$,
\item[(4)]  $c(pt) =1$.
\end{enumerate}
Now, \underline {let us suppose} that there is a similar ``cardinality" on a category $\mathcal {TOP}$ of certain reasonable topological spaces, satisfying the above four properties, except for the condition (1) and (2), 
\begin{enumerate}
\item[(1)'] $X \cong X'$ ($\mathcal{TOP}$-isomorphism) $\Longrightarrow$ $c(X) = c(X')$,
\item[(2)']$c(X) = c(Y) + c(X \setminus Y)$ for a  closed subset  $Y \subset X$.
\item[(3) ] $c(X \times Y) = c(X) \times c(Y)$,
\item[(4) ] $c(pt) =1$.
\end{enumerate} 
If such a ``topological cardinality" exists, then we can show that $c(\mathbb R^1) = -1$, hence $c(\mathbb R^n) = (-1)^n$. Thus, for a finite $CW$-complex $X$, $c(X)$ is exactly the Euler--Poincar\'e characteristic $\chi(X)$. The existence of such a topological cardinality is \emph{guaranteed by the ordinary homology theory}, more precisely 
$$c(X) = \chi_c(X) := \sum (-1)^i \operatorname{dim}_{\mathbb R}H^i_c(X;\mathbb R) = \sum_i (-1)^i \operatorname{dim}_{\mathbb R}H^{BM}_i(X;\mathbb R).$$
 Here $H_*^{BM}(X)$ is the Borel--Moore homology group of $X$. 
 
Similarly let us suppose that there is a similar cardinality on the category $\mathcal V_{\bC}$ of complex algebraic varieties:
\begin{enumerate}
\item[ (1)''] $X \cong X'$ ($\mathcal V_{\bC}$-isomorphism) $\Longrightarrow$ $c(X) = c(X')$,
\item[ (2)''] $c(X) = c(Y) + c(X \setminus Y)$ for a  closed subvariety $Y \subset X$ (i.e., a closed subset in Zariski topology),
\item[(3)  ] $c(X \times Y) = c(X) \times c(Y)$,
\item[(4)  ] $c(pt) =1$.
\end{enumerate} 
We cannot do the same trick as we do for the above $c(\mathbb R^1) = -1$. The existence of such an algebraic cardinality is \emph{guaranteed by Deligne's theory of mixed Hodge structures}. Let $u, v$ be two variables, then the Deligne--Hodge polynomial $\chi_{u,v}$ is defined by 
$$\chi_{u,v}(X) = \sum (-1)^i\operatorname{dim}_{\mathbb C} Gr^p_F Gr^W_{p+q} (H^i_c(X;\mathbb C))u^p v^q.$$ In particular, $\chi_{u,v}(\mathbb C^1) = uv$. The partiuclar case when $u=-y, v = 1$ is the important one for the motivic Hirzebruch class:$\chi_y (X) := \chi_{-y,1}(X) = \sum (-1)^i \operatorname{dim}_{\mathbb C} Gr^p_F (H^i_c(X;\mathbb C))(-y)^p.$ This is called \emph{ $\chi_y$-genus} of $X$. \\

Similarly let us consider the Donaldson--Thomas type invariant of morphisms:
\begin{enumerate}
\item[(1)'''] $X \xrightarrow{f} Y  \cong X' \xrightarrow{f'}  Y$ (isomorphism) $\Longrightarrow$ $\chi^{DT}(X \xrightarrow{f}  Y) = \chi^{DT}(X' \xrightarrow{f'} Y)$,
\item[(2)''']  $\chi^{DT}(X \xrightarrow{f}  Y) =\chi^{DT}(Z \xrightarrow{f|_Z}  Y) + \chi^{DT}(X \setminus Z \xrightarrow{f|_{X \setminus Z}}  Y)$ for a  closed subvariety $Z \subset X$.
\item[(3)'''] $\chi^{DT}(X_1 \times X_2  \xrightarrow {f_1 \times f_2} Y_1 \times Y_2) =\chi^{DT}(X_1 \xrightarrow {f_1} Y_1) \times \chi^{DT}(X_2 \xrightarrow {f_2} Y_2)$,
\item[(4) \, ] $\chi^{DT}(pt) =1$.\\
\end{enumerate} 

So, just like the above two cardinalities or counting $\chi_c(X)$ and $\chi_{u,v}(X)$, we pose the following problem, which is related to the above Problem \ref{bivariant}:
\begin{prob} Is there some kind of bivariant theory $\Theta^?(X \xrightarrow{f}  Y)$ such that
\begin{enumerate}
\item $\chi^{DT}(X \xrightarrow{f}  Y) = \sum_i (-1)^i \op{dim} \Theta^?(X \xrightarrow{f}  Y)?$
\item When $Y$ is smooth, $\Theta(X \xrightarrow{f}  Y)$ is (or should be) isomorphic to the Borel--Moore homology theory $H_*^{BM}(X)$ (which is isomorphic to the Fulton--MacPherson bivariant homology theory $\bH (X \xrightarrow{f}  Y)$).
\end{enumerate}
\end{prob}

\begin{rem}
\begin{enumerate}
\item When $Y$ is smooth, we have $\chi^{DT}(X \xrightarrow{f}  Y) = (-1)^{\op{dim}Y}\chi(X),$ that is
$\chi^{DT}(X \xrightarrow{f}  Y) = (-1)^{\op{dim}Y}\sum_i (-1)^i \op{dim} H_i^{BM}(X)= (-1)^{\op{dim}Y}\sum_i (-1)^i \op{dim} \bH^{-i}(X \xrightarrow{f}  Y).$
 In the above formulation $\chi^{DT}(X \xrightarrow{f}  Y) = \sum_i (-1)^i \op{dim} \Theta^?(X \xrightarrow{f}  Y)$ the sign part $(-1)^i$ should get involve something of the morphism $f$ as well.

\item Even for the identity $X \xrightarrow{\op{id}_X}  X$, since $\chi^{DT}(X) \not = \chi^{DT}(Z) + \chi^{DT}(X \setminus Z)$, the cohomological part $\Theta (X \xrightarrow{\op{id}_X}  X)$ of such a theory (if it existed) does not satisfy the usual long exact sequence for a pair $Z \subset X$, and it should satisfy a modified one so that $$\chi^{DT}(X) = \chi^{DT}(Z \xrightarrow{inclusion}  X) + \chi^{DT}(X \setminus Z \xrightarrow{inclusion}  X)$$ is correct.\\
\end{enumerate} 
\end{rem}

\subsection{A higher class analogue of MNOP conjecture and a generalized MacMahon function}
In \cite{LP} M. Levine and R. Pandharipnade showed the MNOP conjecture \cite{MNOP}, which is nothing but the homomorphism
$$M(q): \Omega^{-3}(pt) \to \bQ[[q]], \, \text{defined by} \, \,  M(q)([X]):= M(q)^{\int_X c_3(T_X \otimes K_X)},$$
where $\Omega^*(X)$ is Levine--Morel's algebraic cobordism \cite{LM} (also see \cite{L3} and \cite{LP})  and
$$M(q) := \prod_{n\leqq 1}\frac{1}{(1-q^n)^n} = 1 + q + 3q^2 + 6q^3 +13q^4 + \cdots$$
is the MacMahon function. A naive question on the above homomorphism $M(q): \Omega^{-3}(pt) \to \bQ[[q]]$ is:

\begin{qu} To what extent could one extend the homomorphism $M(q): \Omega^{-3}(pt) \to \bQ[[q]]$ to a higher dimensional variety $Y$ instead of $Y =pt$ being the point? Namely, could one get the homomorphism
$$M(q): \Omega^*(Y) \to H_*^{BM}(Y) \otimes \bQ[[q]]$$
defined by
$$ M(q)([X \xrightarrow {f} Y]):= M(q)^{f_*\left (c_{\op{dim}X- \op{dim}Y}(T_f \otimes K_f) \cap [X]  \right)}?$$
Here by the construction of the algebraic cobordism $X$ and $Y$ are both smooth, $T_f := T_X - f^*T_Y$ and $K_f:= K_X - f^*K_Y$.
\end{qu}

Note that for $Y =pt$ the above $M(q): \Omega^*(Y) \to H_*^{BM}(Y) \otimes \bQ[[q]]$ is nothing but $M(q): \Omega^{-3}(pt) \to \bQ[[q]]$ in the case when $\op{dim}X = 3$.  The MacMahon function has a combinatorial origin as the generating function for the number of 3-dimensional partitions of size n (as explained in \cite{L3}). It is speculative that the MacMahon function is involved only in the case when $\op{dim}X - \op{dim}Y = 3$. If it were the case, the following more specific problem should be posed: 
\begin{prob} Could one get the homomorphism
$$
M(q): \Omega^{-3}(Y) \to H_*^{BM}(Y) \otimes \bQ[[q]] \,\, 
\text{defined by} \, \, 
M(q)([X \xrightarrow {f} Y]):= M(q)^{f_*\left (c_3(T_f \otimes K_f) \cap [X]  \right)}?
$$
\end{prob}

\begin{rem} Note that the dimension $d$ of an element $[X \xrightarrow {f} Y] \in \Omega^d(Y)$ means that $d = \op{codim}f = \op{dim}Y - \op{dim}X$, hence if $Y = pt$, then $\op{dim}X = 3$ implies that $d =-3$.  Moreover, for a general dimension $d$, say $d < -3$, one should come with some other functions, i.e. ``$d$-dimensional generalized  MacMahon function $\widetilde {M(q)}_d$" such that when $d = -3$ it is the same as the original MacMahon function $M(q)$, i.e. $\widetilde {M(q)}_{-3} = M(q)$. Such a formulation would be useful in Donaldson--Thomas theory for $d$-Calabi--Yau manifolds with $d > 3.$ However, we have to point out that the above function $\widetilde {M(q)}_d$ for the generating function of dimension $d$ partitions is now known to be not correct, although it does appear to be asymptotically correct in dimension four \cite{BBS,MR}. Following ideas from algebraic cobordism as in \cite{LP}, we hope to investigate further in this direction in a future work. 
\end{rem}

\section{Acknowledgements}
Some parts of the paper are based on what we observed or thought during the AIM Workshop ``Motivic Donaldson--Thomas Theory and Singularity Theory" held at the Renyi Institute, Budapest, Hungary, May 7 -- May 11, 2012. We would like to express our sincere thanks to the organizers, Jim Bryan, Andr\'as N\'emethi, Davesh Maulik, J\"org Sch\"urmann, Bal\'azs Szendr\H{o}i, and \'Agnes Szil\'ard for a wonderful organization and to the AIM for financial support for our participation at the workshop. We also would like to thank Bal\'azs Szendr\H{o}i, Dominic Joyce, Laurentiu Maxim, J\"org Sch\"urmann and Vivek Shende for valuable comments and suggestions. \\

\end{document}